\documentclass{amsart}

\usepackage{amssymb, euscript}

\newtheorem{theorem}{Theorem}[section]
\newtheorem{lemma}[theorem]{Lemma}

\newtheorem{definition}[theorem]{Definition}

\theoremstyle{remark}
\newtheorem{remark}[theorem]{Remark}

\numberwithin{equation}{section}

\def\Cal{\mathcal}

\def\R{{\Cal R}}

\def\P{{\Cal P}}
\def\D{{\EuScript{D}}}

\def\S{{\Cal S}}
\def\F{{\Cal F}}

\def\I{{\Cal I}}

\def\W {{\EuScript W_\ell^{\,\a}}}
\def\Wb {{\EuScript W_\ell^{\,\b}}}

\def\sl{\sigma'_\ell vv'\sigma_\ell}
\def\slx{\sigma'_\ell \xi\xi'\sigma_\ell}
\def\s{\EuScript{S}}

\def\tr{{\hbox{\rm tr}}}

\def\Mml{\frM_{m,\ell}}

\def\G{\mathcal{G}}

\def\gnk{G_{n,k}}

\def\f0{f_0}
\def\Fc0{\varphi_0}
\def\rn{\bbr^n}

\def\I_k {I_{-}^{k/2}}
\def\I+k {I_{+}^{k/2}}

\def\vnm{V_{n,m}}

\def\bbr{{\Bbb R}}

\def\bbn{{\Bbb N}}

\def\bbc{{\Bbb C}}

\def\rank{{\hbox{\rm rank}}}

\def\supp{{\hbox{\rm supp}}}

\def\tr{{\hbox{\rm tr}}}

\def\det{{\hbox{\rm det}}}

\def\min{{\hbox{\rm min}}}

\def\Pr{{\hbox{\rm Pr}}}

\def\gnk{G_{n,k}}
\def\gnm{G_{n,m}}

\def\part{\partial}
\def\intl{\int\limits}
\def\b{\beta}

\def\Gam{\Gamma}

\def\a{\alpha}
\def\om{\omega}

\def\del{\delta}
\def\vp{\varphi}

\def\gam{\gamma}

\def\sig{\sigma}

\font\frak=eufm10

\def\fr#1{\hbox{\frak #1}}

\def\frM{\fr{M}}

\def\pl{\P_\ell}

\def\gla{\Gamma_\ell(\a)}

\def\glb{\Gamma_\ell(\b)}

\def\gk{\Gamma_k}

\def\det{{\hbox{\rm det}}}

\def\min{{\hbox{\rm min}}}

\def\p{\P_m}
\def\gm{\Gamma_m}
\def\tr{{\hbox{\rm tr}}}
\def\part{\partial}
\def\intl{\int\limits}
\def\b{\beta}

\def\Gam{\Gamma}

\def\a{\alpha}

\def\cpl{\overline\P_\ell}

\def\sideremark#1{\ifvmode\leavevmode\fi\vadjust{\vbox to0pt{\vss
 \hbox to 0pt{\hskip\hsize\hskip1em
\vbox{\hsize2cm\tiny\raggedright\pretolerance10000
 \noindent #1\hfill}\hss}\vbox to8pt{\vfil}\vss}}}%

\newcommand{\be}{\begin{equation}}
\newcommand{\ee}{\end{equation}}
\newcommand{\bea}{\begin{eqnarray}}
\newcommand{\eea}{\end{eqnarray}}
\newcommand{\Bea}{\begin{eqnarray*}}
\newcommand{\Eea}{\end{eqnarray*}}
\begin{document}

\title[  ]
{Erd\'{e}lyi-Kober integrals on the cone of positive definite
matrices and Radon transforms on Grassmann manifolds}

\author{ E. Ournycheva}
 \address{University of Pittsburgh at Bradford,
 300 Campus Drive, 16701, Bradford, PA, USA }
  \email{elo10@pitt.edu}





\begin{abstract}

We introduce  bi-parametric fractional integrals
 of the Erd\'{e}lyi-Kober type that   generalize
known G{\aa}rding-Gindikin
 constructions  associated to  the cone
of positive definite  matrices. It is proved  that the Radon
transform, which maps  a zonal function on the Grassmann manifold
$\gnm$ of $m$-dimensional linear subspaces of $\bbr^n$ into a function
on the similar manifold $\gnk$, $ 1\le  m<k\le n-1$, is represented
as analytic continuation of the corresponding Erd\'{e}lyi-Kober
integral. This result shows that different Grinberg-Rubin's formulas
for such transforms \cite{GR} have, in fact, a common structure.

 \end{abstract}

\maketitle

\section{Introduction}
\setcounter{equation}{0}

Radon transforms of different kinds have a long history  and
numerous  applications; see \cite{Eh}, \cite{GGG}, \cite{GGV},
\cite{H1}, \cite{Ru5}, and references therein.
 In the present paper, we focus  on important    connection
between Radon transforms on Grassmann manifolds   and  higher rank
fractional integrals. Let $\gnm $ and $\gnk$ be a pair of
Grassmann manifolds of $m$-dimensional and  $k$-dimensional linear
  subspaces  of $\bbr^n$, respectively; $1\le  m<k\le n-1$.
  We use the notation  $\tau_m$ and $\tau_k$ for the respective elements of these Grassmannians. The Radon
 transform of  a function $f$ on $\gnm$
  is a function $ \R f  $ on  $\gnk$
 defined by \be\label{R-GR} (\R f)(\tau_k)=\intl_{ \{
\tau_m : \tau_m \subset \tau_k \} } f(\tau_m)
 d\tau_m,
\ee  where $\tau_k \in \gnk$ and  $d \tau_m$ is  the relevant
probability measure. For $m=1$, a function $f$ on $G_{n,1}$ can be
identified with  an even function on the unit sphere $S^{n-1}
\subset \bbr^n$. In this case, $\R f$ represents the totally
geodesic transform that assigns   to a function $f$ on  the unit
sphere $S^{n-1}$ its integrals over the set of
$(k-1)$-dimensional totally geodesic
submanifolds of $S^{n-1}$. Different aspects of the  Radon transform (\ref{R-GR}) were
investigated  by Gel'fand  and collaborators \cite{GGR}, \cite{GGS},
  Grinberg \cite{Gr1}, Grinberg and Rubin \cite{GR},   Kakehi \cite{Ka1}, \cite{Ka},
  Petrov \cite{P1}, Zhang \cite{Zh},  and  others.

There  is a remarkable connection  between the Radon
transform (\ref{R-GR}) and the following G{\aa}rding-Gindikin
fractional integrals associated to the cone $\pl$ of positive
definite symmetric $\ell \times \ell$ matrices: \bea\label{GGp} (I_{+}^\a
f)(s)  &=&   \frac {1}{\gla} \intl_0^s f(r)\det(s  - r)^{\a-(\ell+1)/2}
dr,\\  \label{GGm}(I_{-}^\a f)(s)  &=& \frac {1}{\gla} \intl_s^{I_\ell}
f(r)\det(r -  s)^{\a-(\ell+1)/2} dr.\eea Here    $I_\ell$ is the identity $\ell\times \ell$
matrix, $$\int_0^s=\int_{\pl \cap (s -\pl)}, \qquad
 \int_s^ {I_\ell}=\int_{(s +\pl)\cap(I_\ell -\pl)},\qquad  s \in \pl;$$ $\gla$ is the  Siegel gamma
 function    (\ref{2.4}); see   \cite{Gar}, \cite{Gi}, \cite{OR2} for more details. For sufficiently good $f$, the integrals $I_{\pm}^\a f$ converge
absolutely if $Re \, \a> (\ell-1)/2$, and extend  to all $\a\in\bbc$ as entire functions of $\a$.

The following result from \cite{GR} is of our main concern. A
function $f$  on  $\gnm$ is canonically identified with right
$O(m)$-invariant function  on the Stiefel manifold $\vnm$ of
$n\times m$  real matrices $v$ satisfying $v'v=I_m$,  $v'$ being the
transpose of $v$.  Abusing notation, we write $f(\tau_m)=f(v)$.
 Fix an integer $\ell$ so that $ 1 \le \ell \le k -m$
and suppose that $f$ is $\ell$-zonal, i.e., $f(v) \equiv f_0(r)$,
  where $r=\sl$, $\sig_\ell=\left[\begin{array} {c} 0 \\ I_\ell   \end{array}
 \right] \in V_{n,\ell}$.
 It is proved in  \cite[Theorem 4.5]{GR}, that for   $ m \ge \ell,$  $\R f$   is represented
by the G{\aa}rding-Gindikin  integral  associated to $\pl$, namely,
 \be \label{R-GG}(\R f)(\tau_k)= c_1 \det(s)^{(\ell -k+1)/2} (I_+^{(k -m)/2}
 \tilde{f}_0)(s),\quad s=\sig'_\ell\Pr_{\tau_k} \sig_\ell, \ee
$$
\tilde f_0(r)=\det(r)^{(m-\ell-1)/2}
f_0(r), \qquad r \in \pl, \qquad c_1 =\frac{\Gam_\ell (k/2)}{\Gam_\ell (m/2)},
$$
$\Pr_{\tau_k}$ denotes the orthogonal projection on $\tau_k$.
 In the case  $m<\ell$,  when  rank$(r)<\ell$, the following formula was obtained in
 \cite[Theorem 4.5]{GR}:
\be \label{R-GG1} (\R f)(\tau_k)=c_2 \intl_0^{I_m} \det(I_m
-r)^{\del} \det( r)^\gam dr \intl_{V_{\ell, m}}
 f_0(s^{1/2} uru's^{1/2}) du, \ee
$$\gam =(\ell-m-1)/2, \quad
  \del=(k
 -m-\ell -1)/2, \quad c_2 =2^{-m}\pi^{-\ell m/2}\frac{\Gam_m (k/2)}{\Gam_m ((k
 -\ell)/2)}.$$

Our aim is to show that right hand sides of  (\ref{R-GG}) and
(\ref{R-GG1}) can be regarded as different forms
of the same  fractional integral, which is introduced below. The
latter resembles well known Erd\'{e}lyi-Kober operators in
fractional calculus \cite{SKM}.

{\bf Plan of the paper and main results.}  Section 2 contains
preliminaries. In Section 3, we  introduce  the following weighted
versions of the G{\aa}rding-Gindikin  integrals: \be \label{EK}
(J_{\pm}^{\a,\b} f)(s)=\frac{\det(s)^{d-\a-\b}}{\glb} (I_{\pm}^\a
g_\b)(s),  \ee $$ g_\b(r)=\det(r)^{\b-d} f(r),\qquad d=(\ell+1)/2.$$
For  $m=1$,    $J_{\pm}^{\a,\b} f$  coincide up to $1/\glb$ with the
classical Erd\'{e}lyi-Kober fractional integrals; see  \cite{SKM}.
 The newly introduced normalizing factor $1/\glb$ is needed
for analytic continuation of $J_{\pm}^{\a,\b} f$ in the
$\b$-variable.  We call (\ref{EK})  {\it fractional integrals of
the Erd\'{e}lyi-Kober type}.

If $f$ is good enough, then integrals  $J_{\pm}^{\a,\b} f$ converge absolutely
for $Re\, \a, Re\, \b>d-1$, and extend as entire functions of $\a$
and $\b$.  We obtain explicit representations  of
$J_{\pm}^{\a, m/2} f$, $m\in\bbn$, provided  that $Re\, \a>d-1$, see (\ref{2.22n}), (\ref{2.23n}). This allows
us to define $J_{\pm}^{\a,\b} f$  for $Re\,\a>d-1 $ and $\b$
 belonging to  the
 Wallach-like  set \cite{FK}  \be\label{Wal} \Wb
=\left\{0, \frac{1}{2}, 1, \frac{3}{2}, \ldots , \frac{\ell-1}{2}
\right\} \cup \left\{\b:\; Re\,\b> \frac{\ell-1}{2} \right \}\ee
 see Definitions \ref{EK+} and  \ref{EK-}.

  In Section 4, we establish connection between the Radon transform
of $\ell$-zonal   functions and integrals (\ref{EK}).

 \begin{theorem}\label{kuku1}
  Let $f$ be an integrable $\ell$-zonal function on $\vnm$, that is, $f(v) \equiv
  f_0(r)$, $r=\sl$.  If $ 1 \leq \ell \le k -m$,
then \be\label{kuku2} (\R f)(\tau_k)= \Gam_\ell (k/2)\,
(J_+^{\frac{k-m}{2},\, \frac{m}{2}} f_0)(s),
 \ee
where   $\tau_k\in G_{n,k}$,  $s=\sig'_\ell\Pr_{\tau_k} \sig_\ell$.
\end{theorem}
Formula (\ref{kuku2}) obviously coincides with (\ref{R-GG})  in the
case $ m \ge \ell$.  As we shall see below (Remark \ref{rem-gr}), it also includes
(\ref{R-GG1}) in the case $ m <\ell$. An analogue  of  Theorem
\ref{kuku1} holds for the dual Radon transform, see Theorem
\ref{dr-ek}.

{\bf Acknowledgement.}  I am   grateful to Professor Boris Rubin for very
helpful comments and discussions.
The research was carried out in the framework of the project ``Higher-rank phenomena in integral
geometry" supported by  NSF grant DMS-0556157.

\section{Preliminaries}\label{sec2}

\subsection{Notation}\label{s2.1}

   Let  $\frM_{n,m}\sim\bbr^{nm}$ be  the
space of $n\times m$  real matrices $x=(x_{i,j})$  with the volume element  $dx=\prod^{n}_{i=1}\prod^{m}_{j=1} dx_{i,j}$. In the following,
$x'$  denotes   the transpose of  $x$, $I_m$
   is the identity $m \times m$
  matrix, and $0$ stands for zero entries. Given a square matrix $a$,  we denote by
  $|a|=\det(a)$
  the determinant of $a$,
  $\tr (a)$ stands for the trace of $a$.  We use  standard
 notations $O(n)$   and $SO(n)$ for the orthogonal group and the
 special orthogonal group of $\bbr^{n}$ with the  normalized
 invariant measure of total mass 1.

Let $\S_\ell \sim \bbr^{\ell(\ell+1)/2}$  be the space of $\ell \times \ell$ real
symmetric matrices $s=(s_{i,j})$
 with the volume element $ds=\prod_{i \le j} ds_{i,j}$.
We denote by  $\pl$   the cone of
positive definite matrices in $\S_\ell$; $\overline\pl$  is the closure of $\pl$, that is the set
of  all positive semi-definite $\ell \times \ell$ matrices.  For $r\in\pl$
($r\in\overline\pl$), we write $r>0$ ($r\geq 0$). Given
 $a, b\in S_\ell$, the inequality $a >b$  means $a - b \in
 \pl$,   the
symbol $\int_a^b f(s) ds$ denotes
 the integral over the set $(a +\pl)\cap (b -\pl)$.
 The group $G=GL(\ell,\bbr)$  of
 real non-singular $\ell \times \ell$ matrices $g$ acts transitively on $\pl$
  by the rule $r \to g rg'$.  The corresponding $G$-invariant
 measure on $\pl$ is \cite[p.
  18]{T} \be\label{2.1}
  d_{*} r = |r|^{-d} dr, \qquad |r|=\det (r), \qquad d= (\ell+1)/2. \ee The  Siegel gamma  function of $\pl$
 is defined by
\be\label{2.4}
 \gla=\intl_{\pl} \exp(-\tr (r)) |r|^{\a } d_*r
 =\pi^{\ell(
 \ell-1)/4}\prod\limits_{j=0}^{\ell-1} \Gam (\a- j/2), \ee
\cite{Gi}, \cite{FK},
 \cite{T}.
  The relevant beta function has the form \be\label{2.6}
  B_\ell (\a ,\b)=\intl_0^{I_\ell} |r|^{\a -d} |I_\ell-r|^{\beta -d} dr=
 \frac{\gla \glb}{\Gam_\ell (\a+\b)}. \ee
These integrals converge absolutely if and only if $Re
 \, \a, Re \, \b >d-1$. For
$1 \le k<\ell, \; k \in \bbn, $ the equality (\ref{2.4}) yields
 \be\label{2.5.2}
\frac{\gla}{\Gam_\ell(\a+k/2)}=\frac{\gk(\a+(k-\ell)/2)}{\gk(\a+k/2)}.\ee
For $r = (r_{i,j}) \in \pl$,  the differential operators  acting in
the $r$-variable are defined by
 \be\label{2.50} D_+ \equiv D_{+,\, r}=
 \det \left ( \eta_{i,j} \, \frac {\partial}{\partial
 r_{i,j}}\right ), \quad
  \eta_{i,j}= \left\{
 \begin{array} {ll} 1  & \mbox{if $i=j$}\\
1/2 & \mbox{if $i \neq j,$}
\end{array}
\right. \ee \be\label{2.51}   D_{-}\equiv D_{-,\, r}=(-1)^{\ell}
D_{+,\, r}.\ee

In the following, all function spaces on a subspace $S$ of $\S_\ell$ are identified with
the corresponding spaces on a  subspace of  $\bbr^{\ell(\ell+1)/2}$. For instance, $\s(S_\ell)$
denotes the Schwartz space    of infinitely differentiable rapidly
decreasing functions;  $\D(S)$ is the
space of functions $f \in C^\infty(S)$ with $\supp f \subset
  S$.

\subsection{Stiefel manifolds}\label{s2.3}

For $n\geq m$, let $\vnm= \{v \in \frM_{n,m}: v'v=I_m \}$
 be  the Stiefel manifold  of orthonormal $m$-frames in $\bbr^n$.
  The group $O(n)$
 acts transitively on $\vnm$  by the left matrix multiplication.
 This is also true for    $SO(n)$  if $n>m$.
We  fix an  invariant measure $dv$ on $\vnm$
  normalized by \be\label{2.16} \sigma_{n,m}
 \equiv \intl_{\vnm} dv = \frac {2^m \pi^{nm/2}} {\gm
 (n/2)}, \ee
\cite[p. 70]{Mu}, and denote  $d_\ast v=\sig^{-1}_{n,m} dv$. The following  statement can be found, e.g., in \cite{Herz},  \cite{Mu},
\cite{FK}.
\begin{lemma}\label{l2.3} {\rm (polar decomposition)} Let $x \in \frM_{n,m}, \; n \ge m$. If  $\rank (x)=
m$, then \be\label{pol} x=vr^{1/2}, \qquad v \in \vnm,   \qquad
r=x'x \in\p,\ee and $dx=2^{-m} |r|^{(n-m-1)/2} dr dv$.
\end{lemma}

\begin{lemma}\label{l2.1} {\rm (bi-Stiefel decomposition).} \ Let $k$,
$m$, and $n$  be positive integers satisfying $1 \le k,\, m   \le n-1$,
 $k+ m \le n.$
 Almost all matrices $v \in \vnm$ can be
represented in the form
 \be\label{herz1}
  v= \left[\begin{array} {cc} a \\ u(I_m -a'a)^{1/2}
 \end{array} \right], \qquad a\in \frM_{k, m}, \quad u \in V_{n-k,m},
 \ee
  so that
   \be\label{2.10}
   \intl_{\vnm} f(v) dv
 =\intl_{0<a'a<I_m}d\mu(a)
 \intl_{ V_{n- k, m}} f\left(\left[\begin{array} {cc} a \\ u(I_m -a'a)^{1/2} \end{array}
 \right]\right) \, du,
\ee
 $$
 d\mu(a)=|I_m-a'a|^\delta da,\quad \del=(n-k-m-1)/2.
 $$

\end{lemma}

For m=1, this  is a well known  bispherical decomposition \cite[pp.
12, 22]{VK2}.
In the higher rank case,  this statement is due to  \cite[p. 495]{Herz} for   $k=m$
 and to Grinberg and Rubin  \cite{GR} for all $k+m \leq n$, see also \cite{OlR},
  \cite{Zh}.

\subsection{The Laplace transform}
  Let $z=\sigma+i\om$, $\sigma \in \pl$,
$\om\in \S_\ell$,  be a complex symmetric
matrix.  Suppose that  $f$ is  a locally integrable function on
$\S_\ell$ satisfying
 $f(r)=0$ if $r\notin \cpl$, and $\exp(-\tr(\sigma_0
r))f(r)\in L^1(\S_\ell)$ for some $\sigma_0\in\pl$. The integral
\be\label{3.11}
 (Lf)(z)=\intl_{\pl} \exp(-\tr(zr))f(r)dr
\ee is called the Laplace transform  of $f$.
This integral is  absolutely convergent in the (generalized) half-plane
$Re\;z>\sigma_0$.
 Let \be\label{fts}
(\F g)(\om)=\intl_{\S_\ell} \exp(\tr(i\om s)) g (s) ds,\qquad
\om\in\S_\ell, \ee
 be the Fourier transform  of  a function $g$  on $\S_\ell$. Then
$
(Lf)(z)=(\F g_\sigma)(-\om),
$
where $g_\sigma(r)=\exp(-\tr(\sigma  r))f(r)\in L^1(\S_\ell)$ for
$\sigma >\sigma_0$. Thus, all properties of the Laplace transform
are obtained from the general Fourier transform theory for
Euclidean spaces  \cite{Herz}, \cite[ p. 126]{Vl}.
 The following uniqueness result for the
 Laplace transform  follows from injectivity of
the Fourier transform of tempered distributions.
\begin{lemma}\label{lap}
  If $f_1 (r)$ and $ f_2 (r)$ satisfy $ \exp(-\tr(\sigma_0
r))f_j(r)\in L^1(\P_\ell), \quad j=1,2, $
 for some $ \sigma_0 \in \pl$,  and $(Lf_1)(z)=(Lf_2)(z)$ whenever
$Re\;z>\sigma_0$, then $f_1(r)= f_2(r)$  almost everywhere on $
\S_\ell$.
\end{lemma}

\subsection{G{\aa}rding-Gindikin distributions}
Let  $f$ belong  to the Schwartz space $\s(\S
_\ell)$.  The  G{\aa}rding-Gindikin
distribution associated to the cone $\pl$ is defined by   \be\label{gdd}
\G_\a (f)=\frac {1}{\gla} \intl_{\pl} f(r)|r|^{\a-d} dr,\qquad
d=(\ell+1)/2. \ee   The integral (\ref{gdd}) converges absolutely for $Re\,\a >
d-1$ and
 admits  analytic
continuation as an entire function of $\a$ so that $
\G_0(f) =f(0)$.  The integrals of half-integral order  have the form  \be\label{er} \G_{m/2}
(f)=\pi^{-\ell m/2}\intl_{\Mml} f(\om'\om)d\om, \quad m=1,2,\dots, \ee
see \cite[pp. 132--134]{FK}, \cite{OR3}.

 \subsection{The G{\aa}rding-Gindikin fractional integrals}
Let
 $Q=\{r \in \pl: 0<r<I_\ell \}$     be the ``unit interval'' in $\pl$.
 For $f\in L^1(Q)$ and $Re \, \a > d-1$, the G{\aa}rding-Gindikin integrals are defined by
\bea\label{3.1} (I_{+}^\a f)(s)  &=&   \frac {1}{\gla} \intl_0^s
f(r)|s  -  r|^{\a-d} dr,\\  \label{3.2}(I_{-}^\a f)(s)  &=& \frac
{1}{\gla} \intl_s^{I_\ell}   f(r)|r -  s|^{\a-d} dr,\eea where $s \in
Q$.  Both integrals are absolutely convergent.
 For $f\in\D(Q)$ and
$Re \, \a \le d-1$,
 the analytic continuation of the integrals (\ref{3.1}) and (\ref{3.2}) can be defined by   \be\label{brr} (I_{\pm}^\a
f)(s)=(I_{\pm}^{\a+j}D_{\pm}^j f)(s) \quad \text{\rm if} \quad
d-1-j<Re \, \a \le d-j; \quad j=1,2, \ldots\; ,\ee where $D_{\pm}$ are differential operators (\ref{2.50}), (\ref{2.51}).

\begin{lemma} For $f\in\D(Q)$ and $\a\in\bbc$,
\be\label{3.7}(I_{\pm}^{\a}f)(s)=(I_{\mp}^{\a}g)(I_\ell-s),\qquad g(r)=
f(I_\ell-r).
 \ee
\end{lemma}
\begin{proof}
Since the integrals $(I_{+}^\a f)(s)$ and $ (I_{-}^\a f)(s)$  are
entire functions of $\a$, it suffices to prove (\ref{3.7}) for $Re
\, \a  >  d-1$. This can be easily done by changing variables $r\to
I_\ell-r$.
\end{proof}

\begin{theorem}\label{th-GGl}
If $f\in\D(Q)$, then for all $m \in \bbn$, \bea
(I_{+}^{m/2}f)(s) \label{2.22}&=& \pi^{-\ell
m/2}\intl_{\{\om\in\Mml:\;\om'\om<s\}}
f(s-\om'\om)d\om,\\(I_{-}^{m/2} f)(s) \label{2.23} & =&\pi^{-\ell
m/2}\intl_{\{\om\in\Mml:\;\om'\om<I_\ell-s\}} f(s+\om'\om)d\om .\eea
Moreover, \be\label{2.32}  (I_{\pm}^{0} f)(s)=f(s). \ee
\end{theorem}
\begin{proof}

Formulas (\ref{2.22}) and $(I_{+}^{0} f)(s)=f(s)$ are  verified  in
\cite{OR3} for arbitrarily $s\in\pl$; (\ref{2.23}) and the second
equality in (\ref{2.32})  follow from (\ref{3.7}) and the
corresponding properties of the integral $I_+ ^\a f$.
\end{proof}

\begin{theorem}\label{l3.4} If $f \in L^1(Q)$ and $m\in \bbn$, then $(I_{\pm}^{m/2}
f)(s)$ converge   absolutely  for almost all $s \in Q$.
\end{theorem}
\begin{proof}
It was shown in \cite{OR3} that for nonnegative $f$ and  every  $a \in \pl$,
\be \int_0^{a} (I_{+}^{m/2}f)(s)ds
  \le \frac{\Gam_\ell(d)}{\Gam_\ell(m/2+d)}\; |a|^{m/2}  \intl_0^{a} f(r) dr\, . \nonumber \ee
 This proves that  $(I_{+}^{m/2}
f)(s)$ is    absolutely convergent   for almost all $s \in Q$ provided that  $f \in L^1(Q)$. The statement for the right-sided integral is a consequence of  (\ref{3.7}).
\end{proof}

According to Theorems \ref{th-GGl}  and \ref{l3.4}, the G{\aa}rding-Gindikin fractional integrals can be defined  for arbitrary
integrable functions $f$  and  $\a$ belonging  to the
Wallach  set \be\label{wal} \W =\left\{0,
\frac{1}{2}, 1, \frac{3}{2}, \ldots , \frac{\ell-1}{2} \right\} \cup
\left\{\a:\; Re\,\a> \frac{\ell-1}{2} \right \}.\ee

\section{  The generalized Erd\'{e}lyi-Kober fractional integrals}

Let $f$ be  a function defined  on the   ``unit interval'' $\bar Q=\{r \in
\cpl: 0\leq r\leq I_\ell \}$ in $\cpl$. For  $Re \, \a,\; Re \, \b
> d-1$, we introduce  the generalized Erd\'{e}lyi-Kober fractional integrals    \bea\label{EKp}
(J_{+}^{\a,\b} f)(s) & =& \frac {|s|^{d-\a-\b}}{\gla\glb} \intl_0^s
f(r)|r|^{\b-d}|s  - r|^{\a-d} dr,\\\label{EKm} (J_{-}^{\a,\b} f)(s)
&=& \frac {|s|^{d-\a-\b}}{\gla\glb} \intl_s^{I_\ell} f(r)|r|^{\b-d}|r  -  s|^{\a-d}
dr, \eea  where $s\in Q$,  $ d=(\ell+1)/2$.  Since  \be\label{EK-GG}(J_{\pm}^{\a,\b}
f)(s)=\frac{|s|^{d-\a-\b}}{\glb} (I_{\pm}^\a g_\b)(s),\qquad g_\b(r)=|r|^{\b-d}
f(r), \ee
  these integrals
 converge   absolutely  for almost all $s \in Q$ whenever $g_\b\in L^1 (Q)$.
Suppose that  for fixed  $\b$ satisfying   $Re \, \b
> d-1$,  the function  $g_\b$ is infinitely differentiable and
 supported in $Q$. Then the integrals   $J_{\pm}^{\a,\b} f$
admit  analytic continuation as entire functions of $\a$ so that
$(J_{\pm}^{0,\b} f)(s)=f(s)/\glb$.

Our next  goal is to show that  integrals (\ref{EKp}), (\ref{EKm})  extend  analytically as entire functions of $\b$ and   obtain   an explicit form  of
$J_{\pm}^{\a, m/2} f$ for  $m\in\bbn$
provided that $Re \, \a> d-1$.

 \begin{theorem}
 Let    $f$ be an   infinitely differentiable function
   supported in   $Q$. The integrals  $J_{\pm}^{\a,\b} f$ are entire functions of $\a$ and $\b$.  For all $m \in \bbn$,
and $Re \, \a> d-1$, \bea (J_{+}^{\a,\,m/2} f)(s) \label{2.22n}&=&
\frac {\pi^{-\ell
m/2} |s|^{d-\a-\b}}{\gla}\!\!\!\!\!\!\!\!\!\!\intl_{\{\om\in\Mml:\;\om'\om<s\}}\!\!\!\!\!\!\!\!\!\!|s
- \om'\om|^{\a-d} f(\om'\om)d\om, \\
 (J_{-}^{\a,\,m/2} f)(s) \label{2.23n} &=&\frac {\pi^{-\ell
m/2}|s|^{d-\a-\b}}{\gla}\!\!\!\!\!\!\!\!\!\!\intl_{\{\om\in\Mml:\;\om'\om<I_\ell-s\}}\!\!\!\!\!\!\!\!\!\!|s
+ \om'\om|^{\a-d} f(\om'\om)d\om .\eea
\end{theorem}
\begin{proof}
 We  apply the Laplace transform technique  to prove (\ref{2.22n}). One can assume that $f$ is a function  on $\S_\ell$ which is identically  zero outside $Q$, so that   $f\in \s(\S
_\ell)$.  Denote \be\label{Jt}(\tilde J_{+}^{\a,\,\b} f)(s)=|s|^{\a+\b-d}(J_{+}^{\a,\,\b} f)(s).\qquad  s\in \pl\ee  It is known (see, e.g., \cite{OR3}) that for $f\in\D(\pl)$, $s\in \pl$,  and $\a\in\bbc$,
 \be\label{3.9.1}
(LI_{+}^{\a}f)(z)=\det(z)^{-\a}(Lf)(z), \qquad Re\;z>0. \ee
Here,  $\det(z)^{-\a}=\exp(-\a \log\det(z))$, where  the
branch of $\log\det(z)$ is chosen so that  $\det(z)=|\sig|$ for
real $z=\sigma\in\p$. Hence, for $Re\,\b>d-1$ and $Re\;z>0$,
\be\nonumber
(L\tilde J_{+}^{\a,\b}f)(z)=\frac{\det(z)^{-\a}}{\glb}(Lg_\b)(z)= \frac{\det(z)^{-\a}}{\glb} \intl_{\pl} |r|^{\b-d} \exp(-\tr(zr)) f(r)dr. \ee
Denote $F_{z}(r)= \exp(-\tr(zr))f(r)$. Then
$
(L\tilde J_{+}^{\a,\b}f)(z)=\det(z)^{-\a}\G_\b (F_{z}),
$
where $\G_\b (F_{z})$ is the  G{\aa}rding-Gindikin distribution (\ref{gdd}), which is entire function of $\b$. By (\ref{er}), for $m=1,2,\dots$, we obtain
\be
(L\tilde J_{+}^{\a, m/2}f)(z)=\pi^{-\ell m/2} \det(z)^{-\a} \intl_{\Mml} \exp(-\tr(z\om'\om)) f(\om'\om)d\om.
\ee
On the other hand, the Laplace transform of the integral
$$I(s)=\frac {\pi^{-\ell
m/2} }{\gla}\!\!\!\!\!\!\!\!\!\!\intl_{\{\om\in\Mml:\;\om'\om<s\}}\!\!\!\!\!\!\!\!\!\!|s
- \om'\om|^{\a-d} f(\om'\om)d\om$$
can  be evaluated as follows
\bea\nonumber
(LI)(z)&=&\frac {\pi^{-\ell
m/2} }{\gla} \intl_{\Mml} f(\om'\om)d\om \intl_{ \om'\om}^\infty \exp(-\tr(z s))|s
- \om'\om|^{\a-d}  ds\\\nonumber &=&\frac {\pi^{-\ell
m/2} }{\gla} \intl_{\Mml} \exp(-\tr(z \om'\om)) f(\om'\om)d\om \intl_{ \pl} \exp(-\tr(z s))|s|^{\a-d}  ds.
\eea
By  the  the well known  formula
$$\intl_{\pl} \exp(-\tr(zr))|r|^{\a-d} dr=\gla
\det(z)^{-\a},\qquad Re\;\a>d-1, $$ see, e.g., \cite[p. 479]{Herz}, we obtain
$$(LI)(z)=\pi^{-\ell m/2} \det(z)^{-\a} \intl_{\Mml} \exp(-\tr(z\om'\om)) f(\om'\om)d\om.$$
According to the uniqueness property of the Laplace transform, it follows that
$$
(\tilde J_{+}^{\a, m/2}f)(s)=\frac {\pi^{-\ell
m/2} }{\gla}\!\!\!\!\!\!\!\!\!\!\intl_{\{\om\in\Mml:\;\om'\om<s\}}\!\!\!\!\!\!\!\!\!\!|s
- \om'\om|^{\a-d} f(\om'\om)d\om.
$$
This proves (\ref{2.22n}). The statement for the right-sided integral  is a consequence of
 (\ref{EK-GG}),  (\ref{3.7}), and  (\ref{2.22n}).

\end{proof}

\begin{theorem}\label{th-ex} Let  a function $f$ on $\bar Q$ satisfy  the
inequality  \be\label{EK-ex}
\intl_{\{\om\in\frM{m,\ell}:\,\om'\om<I_\ell\}} |f(\om'\om)| d\om
<\infty, \qquad   m \in\bbn. \ee
    Then the integrals  $(J_{\pm}^{\a,\,m/2}
f)(s)$ converge  absolutely  for $ Re\;\a
>d-1$ and almost all $s\in Q$.
\end{theorem}

\begin{proof}  It suffices to show that the
integrals
 $ \int_0^{I_\ell}(\tilde J_{\pm}^{\a,\,m/2} f)(s) ds$ are  finite for
 nonnegative
  $f $, where $\tilde J_{\pm}^{\a,\,m/2} f$ are defined by (\ref{Jt}).  For the left-sided integral,  changing the order of integration yields
 \bea \int_0^{I_\ell} (\tilde J_{+}^{\a,\,m/2}
f)(s) ds &=&\frac {\pi^{-\ell m/2}}{\gla}
  \intl_{\om'\om<I_\ell} f(\om'\om)d\om\intl_{\om'\om}^{I_\ell}  |s-\om'\om|^{\a-d} ds
  \nonumber \\ &=&\frac {\pi^{-\ell m/2}}{\gla}
  \intl_{\om'\om<I_\ell} f(\om'\om) d\om\intl_0^{I_\ell-\om'\om} |r|^{\a-d} dr  \nonumber \\
  &=&\frac {\pi^{-\ell m/2}\Gam_\ell (d)}{\Gam_\ell(\a+d)} \intl_{\om'\om<I_\ell} |I_\ell-\om'\om|^{\a}f(\om'\om) d\om
  \nonumber\\
&\leq &\frac {\pi^{-\ell m/2}\Gam_\ell (d)}{\Gam_\ell (\a+d)}
\intl_{\om'\om<I_\ell} f(\om'\om) d\om <\infty, \nonumber \eea
as required. Here,
we used the substitution $r= a^{1/2}sa^{1/2}$, $dr=|a|^{d}ds$ \cite[
pp. 57--59]
 {Mu}, to evaluate the integral
 $$ \intl_0^{a} |r|^{\a-d} dr =|a|^{\a}\intl_0^{I_\ell } |s|^{\a-d} ds=B_\ell(\a,d)|a|^{\a},\qquad a=I_\ell-\om'\om.$$
 The proof  for $J_{-}^{\a,\,m/2} f$ is the same.

\end{proof}

\begin{remark}
For $m\geq \ell$, condition (\ref{EK-ex}) is equivalent to $|r|^{m/2-d}f(r)\in L^1(Q)$.

\end{remark}
Theorem \ref{th-ex}  allows us to define integrals
$J_{\pm}^{\a,\,\b}f$ for $Re\a>d-1$ and all $\b$ belonging to the Wallach  set (\ref{Wal}) provided that   $|r|^{\b-d}
f(r)\in L^1(Q)$ when $Re \, \b > d-1$,  and $f$ satisfies (\ref{EK-ex})
when  $\b=m/2$, $m=1,2,\dots \ell-1$.

\begin{definition}\label{EK+}
The left-sided   Erd\'{e}lyi-Kober type  fractional integral is defined by

 \be\nonumber (J_{+}^{\a,\b} f)(s) = \left \{
\begin{array} {ll}\displaystyle{\frac {|s|^{d-\a-\b}}{\gla\glb} \intl_0^s
f(r)|r|^{\b-d}|s  - r|^{\a-d} dr}
 & \mbox{ if   $Re \, \b > d-1$}, \\
{} \\
 \displaystyle{\frac {\pi^{-\ell
m/2} |s|^{d-\a-\b}}{\gla}\!\!\!\!\!\!\!\!\intl_{\{\om\in\Mml:\;\om'\om<s\}}\!\!\!\!\!\!\!\!|s
- \om'\om|^{\a-d} f(\om'\om)d\om} & \mbox { if   $\b=m/2$}. \\
\end{array}
\right.
 \ee
 \end{definition}

 \begin{remark}
 For $0<m<\ell$, the second  integral in Definition \ref{EK+}  can be rewritten as follows:
 \be\label{J+gr}
 (J_{+}^{\a, m/2} f)(s)=c \intl_0^{I_m} |q|^{(\ell-m-1)/2} |I_m-q|^{\a-d}\; dq\intl_{V_{\ell,m}} f(s^{1/2}uqu's^{1/2} ) \; du,
 \ee
 $$
 c= 2^{-m} \pi^{-\ell
m/2}/ \gla. $$

 \end{remark}

 \begin{definition}\label{EK-}
The right-sided Erd\'{e}lyi-Kober type fractional integral is defined as follows
\be\nonumber (J_{-}^{\a,\b} f)(s) = \left \{
\begin{array} {ll}\displaystyle{\frac {|s|^{d-\a-\b}}{\gla\glb} \intl_s^{I_\ell} f(r)|r|^{\b-d}|r  -  s|^{\a-d}
dr}
 & \mbox{ if   $Re \, \b > d-1$}, \\
{} \\
 \displaystyle{\frac {\pi^{-\ell
m/2} |s|^{d-\a-\b}}{\gla}\!\!\!\!\!\!\!\!\intl_{\{\om\in\Mml:\;\om'\om<I_\ell-s\}}\!\!\!\!\!\!\!\!|s
+ \om'\om|^{\a-d} f(\om'\om)d\om} & \mbox { if   $\b=m/2$}. \\
\end{array}
\right.
 \ee
\end{definition}


\section {The Radon transform on Grassmannians}

\subsection{Definitions}

Let  $1\leq m<k\leq n-1$.  Given  a pair of Grassmann manifolds $\gnm$ and   $\gnk$,  the Radon transform of a  function $f(\tau_m)$ on $\gnm$ is defined  by (\ref{R-GR}). The corresponding
dual   transform of    a function $\vp(\tau_k)$ on $\gnk$ is
 \be\label{R-d} (\R^* \vp)(\tau_m)=\intl_{ \{
\tau_k : \tau_m \subset \tau_k \} } \vp(\tau_k)
 d\tau_k, \qquad \tau_m \in \gnm.
\ee
Let $\tau_m=\{x\in\rn:\, \zeta' x =0,   \; \zeta \in V_{n,n-m} \},$
$\tau_k=\{x\in\rn:\, \xi' x =0,   \; \xi \in V_{n,n-k} \}.$
Denote by $v= \zeta^\perp$  the orthogonal complement of
$\zeta$.  To  give  precise meanings to  integrals  (\ref{R-GR}), (\ref{R-d}), we use the following parameterizations:
\be \nonumber \tau_m\equiv\tau_m(v),\; v \in V_{n,m};\qquad
\tau_k\equiv\tau_k(\xi),\; \xi \in V_{n,n-k}.\ee The functions  $f(\tau_m)$ on $\gnm$ and $\varphi(\tau_k)$ on $\gnk$
are  identified  with the right-invariant functions on the Stiefel
manifolds $\vnm$ and $V_{n,n-k}$, respectively.
Condition $\tau_m\subset\tau_k$ is equivalent to $\xi'v=0$. The set
of all $v$ satisfying the last equation is represented as $v=g_\xi
 \left[\begin{array} {c} \om \\ 0   \end{array} \right]$,
 $\om\in V_{k,m}$,
where   $g_\xi$ is an arbitrary rotation with the
 property \be\label{gxi} g_\xi  \xi_0 =\xi, \quad \xi_0= \left[\begin{array} {c}  0 \\  I_{n-k}
\end{array} \right] \in V_{n,n-k}. \ee This observation leads to the following.

 \begin{definition}
 The Radon transform of a right-invariant function $f(v)$ on $\vnm$ is defined
by the formula
 \be \label{rad}(\R f)(\xi)=\intl_{V_{k,m}} f \left ( g_\xi
 \left[\begin{array} {c} \om \\ 0   \end{array} \right]\right) \,
 d_*\om, \qquad \xi \in V_{n,n-k}. \ee
 \end{definition}

 \begin{remark}
 Definition (\ref{rad}) of the Radon transform  differs from  that in \cite{GR} by the parametrization of the plane $\tau_k$.
 Specifically,  in \cite{GR},
  $(\R f)(\tau_k)\equiv (\R f)(\xi^\perp) $.
 \end{remark}

 Let  $\gam_v$ be   an arbitrary rotation with the
 property \be \label{gamv}\gam_v  v_0 =v, \quad v_0= \left[\begin{array} {c}  0 \\  I_{m}
\end{array} \right] \in \vnm. \ee The set of all $\xi$ that obey $\xi' v=0$ has the form   $\xi=\gam_v
 \left[\begin{array} {c} u \\ 0   \end{array} \right], \quad u\in V_{n-m,n-k} $.
\begin{definition}
The   dual Radon transform of a right-invariant
function $\vp(\xi)$ on $V_{n, n-k}$  is defined by  \be
\label{dual-rad}(\R^* \vp)(v)=\intl_{V_{n-m,n-k}} \vp \left ( \gam_v
 \left[\begin{array} {c} u \\ 0   \end{array} \right]\right) \,
 d_*u,\qquad v\in\vnm. \ee
 \end{definition}

 The expressions   (\ref{rad}) and (\ref{dual-rad}) are independent of the choice of
rotations $ g_\xi : \xi_0 \to \xi$, and $\gam_v : v_0 \to v$,
respectively.

\begin{lemma}
The duality relation

 \be\label{4.3}
\intl_{\vnm} f(v) (\R^*\vp)(v) d_*v=
\intl_{V_{n, n-k}} (\R f)(\xi) \vp(\xi) d_*\xi
 \ee
is valid  provided that either side of this equality is finite for
$f$ and $\vp$ replaced by $|f|$ and $|\vp|$, respectively.
\end{lemma}
\begin{proof}
By (\ref{dual-rad}), the left-hand side of (\ref{4.3})  equals
\bea\nonumber
l.h.s.&= &\intl_{\vnm} f(v)  d_*v \intl_{V_{n-m,n-k}} \vp \left ( \gam_v
 \left[\begin{array} {c} u \\ 0   \end{array} \right]\right) \,
 d_*u\\\nonumber
 &= & \intl_{SO(n-m)}   d\gam \intl_{SO(n)}  f(\b v_0) \vp \left (\b \gam_{v_0}
 \left[\begin{array} {cc} \gam & 0 \\ 0 & I_{m}   \end{array} \right] \tilde \xi_0 \right) \,
 d \b,
\eea
where $\tilde \xi_0=\left[\begin{array} {c} I_{n-k} \\ 0   \end{array} \right]\in V_{n,n-k}$.  The change of variables
$
\b \gam_{v_0}
 \left[\begin{array} {cc} \gam & 0 \\ 0 & I_{m}   \end{array} \right]\to \b
$ yields
\bea\nonumber
l.h.s.&=&\intl_{SO(n-m)} \,
 d\gam \intl_{SO(n)} \vp (\b
  \tilde \xi_0 ) \,  f\left (\b \left[\begin{array} {cc} \gam' & 0 \\ 0 & I_{m}   \end{array} \right] \gam'_{v_0} v_0\right)  d\b \\
 &=& \nonumber \intl_{SO(n)} \vp (\b
  \tilde \xi_0 ) \,  f\left (\b v_0\right)  d\b.
 \eea
 Similarly, according to (\ref{rad}), the right-hand side of (\ref{4.3}) is  evaluated as follows
 \bea\nonumber
r.h.s.&= &\intl_{V_{n,n-k}} \vp(\xi)  d_*\xi\intl_{V_{k,m}} f \left ( g_\xi
 \left[\begin{array} {c} \om \\ 0   \end{array} \right]\right) \,
 d_*\om\\\nonumber
 &= & \intl_{SO(n)} \vp (\b
  \xi_0 ) \,  f\left (\b \tilde  v_0\right)  d\b=\intl_{SO(n)} \vp (\b
  \tilde \xi_0 ) \,  f\left (\b v_0\right)  d\b,
  \eea
  where $\tilde v_0=\left[\begin{array} {c} I_{m} \\ 0   \end{array} \right]\in V_{n,m}$.  This proves the duality relation.
\end{proof}

\begin{theorem}
The Radon transform $\R f$ and the dual Radon transform  $\R^* \vp$ are well defined almost everywhere
on $V_{n, n-k}$ and $\vnm$, respectively,   for any integrable functions  $f$
 and $\vp$.
\end{theorem}
\begin{proof} It follows from (\ref{4.3}) that
\bea
\intl_{V_{n, n-k}} (\R f)(\xi) d_*\xi &=&\intl_{\vnm} f(v)  d_*v,\\
\intl_{\vnm} (\R^*\vp)(v) d_*v &=&
\intl_{V_{n, n-k}}  \vp(\xi) d_*\xi.
 \eea
 Hence, the integrals  $(\R f)(\xi)$, $(\R^*\vp)(v)$ exist  for almost all $\xi$ and $v$, respectively.

\end{proof}

\subsection {Radon transform and its dual of invariant functions}

The notion of $\ell$-zonal function  on the Stiefel manifold  was introduced in \cite{GR}.
\begin{definition}\label{l-zon}
 Let $1\leq \ell\leq n-1$, $\rho \in O(n-\ell), \; g_\rho= \left[\begin{array} {ccl} \rho
 & 0  \\ 0 & I_\ell \end{array} \right] \in O(n).$ A function
 $f(v)$ on $\vnm \; $  is called $\ell$-{\it
 zonal} if  $f(g_\rho v)=f(v) \; $  for
 all $\rho \in O(n-\ell)$.
\end{definition}
\begin{lemma}   \cite[p. 798]{GR}  Let $m+\ell \le n$,  $ m \ge \ell$. A function $f(v)$
on $\vnm$ is
 $\ell$-zonal and $O(m)$ right-invariant (simultaneously)
 if and only if there is a function $f_0$ on $\pl$ such that
 $f(v) \stackrel{\rm a.e.}{=}f_0(r), \; r= \sl=\sig'_\ell \Pr_{v
 } \sig_\ell$, $\sig_\ell =
 \left[\begin{array} {c} 0 \\ I_\ell \end{array} \right] \in V_{n,
 \ell}$. Here, $\Pr_v\sig_\ell$ is the
orthogonal projection  of $\sig_\ell$ onto  $v$.
\end{lemma}

The following theorem establishes connection between the Radon transform of a function $f$ of the form $f(v)\equiv f_0(r)$, $r=\sl$, and the Erd\'{e}lyi-Kober type fractional integrals associated to the cone $\pl$. Note that for $m\ge  \ell$, the function $f_0$ is defined on  $\pl$, and the integral $J_+^{(k-m)/2,\, m/2} f_0$ exists in the usual sense, cf. (\ref{EKp}). For $m<\ell$, $f_0$ is a function on the boundary $\partial \pl$, and $J_+^{(k-m)/2,\, m/2} f_0$ is understood  in the sense of analytic continuation (\ref{2.22n}).

\begin{theorem}
 Let $f(v)$ be an integrable function on $\vnm$. Suppose that
$f(v)$ has the form $f(v) \equiv f_0(r)$, where  $r=\sl\in\overline\pl$, and
denote  $s=I_\ell- \slx=\sig'_\ell \Pr_{\xi^\perp} \sig_\ell$. If $
1 \le \ell \le k -m$, then \be\label{R-zon} (\R f)(\xi)=
\Gam_\ell (k/2) (J_+^{(k-m)/2,\, m/2} f_0)(s),
 \ee
 where $J_+^{(k-m)/2,\,m/2} $ is the the Erd\'{e}lyi-Kober type operator.
\end{theorem}
\begin{proof}
By (\ref{rad}), \be(\R f)(\xi)=\intl_{V_{k,m}} f_0 (y_\om y'_\om)\,
 d_*\om, \qquad y_\om=\sig'_\ell g_\xi \left[ \begin{array} {c} \om \\ 0   \end{array}
\right]. \ee  Denote \[ a=g^{-1}_\xi \sig_\ell=\left[ \begin{array} {c} a_1 \\
 a_2
 \end{array} \right], \qquad a_1 \in \frM_{k,\ell}, \quad  a_2 \in
 \frM_{n-k,\ell},
\]
so that  $y_\om=a'_1 \om$. By  Lemma 2.1, $a_1=us^{1/2}$, where
$u\in V_{k,\ell}$ and  $$s=
a'_1 a_1=\sig'_\ell g_\xi \left[\begin{array} {cc} I_{k} & 0 \\
0
 & 0 \end{array} \right] g_\xi^{-1} \sig_\ell = \sig'_\ell \Pr_{\xi^\perp}
 \sig_\ell.
 $$  Hence,
 $$(\R f)(\xi)=\intl_{V_{k,m}} f_0 ( s^{1/2}u'\om\om' u s^{1/2})\,
 d_*\om=\intl_{V_{k,m}} f_0 ( s^{1/2}u_0'\om\om' u_0 s^{1/2})\,
 d_*\om, $$ where
$u_0 =
 \left[\begin{array} {c} I_\ell \\ 0 \end{array} \right] \in V_{k,
 \ell}$.
  Applying  the bi-Stiefel decomposition
$$
  w= \left[\begin{array} {cc} b \\ v_1(I_m -b'b)^{1/2}
 \end{array} \right], \qquad b\in \frM_{\ell, m}, \quad v_1 \in
 V_{k-\ell,m},
 $$
according to (\ref{2.10}),  we obtain
\bea   \nonumber
 (\R f)(\xi)
 &=&\frac{ \sig_{k-\ell,m}}{\sig_{k,m}}\!\!\!\intl_{\{b\in \frM_{\ell, m}:\;0<b'b<I_m\}}\!\!\!\!\!\!\!\!\!\!\!\!
 |I_m -b'b|^{(k-m-\ell-1)/2} f_0 ( s^{1/2}bb' s^{1/2})
 db\\\nonumber &=&\frac{ \sig_{k-\ell,m}}{\sig_{k,m}}\!\!\!\!\intl_{\{y\in \frM_{ m,\ell}:\;0<y'y<I_\ell\}}\!\!\!\!\!\!\!\!\!\!\!\!
 |I_\ell -y'y|^{(k-m-\ell-1)/2} f_0 ( s^{1/2}y'y s^{1/2})
   dy \eea
The change of variables $z=y s^{1/2}$, $dz=|s|^{m/2} dy$,  yields \bea   \nonumber
 (\R f)(\xi)
 &=&\frac{ \sig_{k-\ell,m} |s|^{-m/2}}{\sig_{k,m}}\intl_{\{z\in \frM_{ m,\ell}:\;0<z'z<s\}}\!\!\!\!\!\!\!\!\!\!\!\!
 |I_\ell -s^{-1/2}z'zs^{-1/2}|^{(k-m-\ell-1)/2} f_0 ( z'z)
 dz\\\nonumber &=&\frac{ \sig_{k-\ell,m}|s|^{(\ell-k+1)/2}}{\sig_{k,m}}\intl_{\{z\in \frM_{ m,\ell}:\;0<z'z<s\}}\!\!\!\!\!\!\!\!\!\!\!\!
 |s -z'z|^{(k-m-\ell-1)/2} f_0 ( z'z)
   dz .\eea
Owing to (\ref{2.22n}), this gives
$$
      (\R f)(\xi)=  c
\;   (J_+^{\frac{k-m}{2},\, \frac{m}{2}} f_0)(s),
      $$  where
      $$ c=\frac{ \pi^{\ell m/2}\sig_{k-\ell,m}\Gam_{\ell}(\frac{k-m}{2})}{\sig_{k,m}}
      =\frac{ \Gam_{\ell}(\frac{k-m}{2})\Gam_m (\frac{k}{2})}{\Gam_m (\frac{k-\ell}{2})}\, .
      $$
By (\ref{2.5.2}),
\be\label{gam-new}
 \frac{\Gam_m(\frac{k-\ell}{2})}{\Gam_m (\frac{k}{2})}=  \frac{\Gam_\ell(\frac{k-m}{2})}{\Gam_\ell
 (\frac{k}{2})},
\ee
which yields   $c=\Gam_\ell (k/2)$.

\end{proof}

\begin{remark}\label{rem-gr}
It follows from (\ref{J+gr}) and (\ref{gam-new}) that for $m=1,2,\dots \ell-1$, formula (\ref{R-zon}) coincides with (\ref{R-GG1} ).
\end{remark}

Let us consider  the dual Radon transform.
\begin{theorem}\label{dr-ek}

 Let $\vp(\xi)$ be an integrable function on $V_{n,n-k}$. Suppose that
$\vp(\xi)$ has the form $\vp(\xi) \equiv \vp_0(s)$, where
$s=\slx\in\overline\pl$, and denote  $r=I_\ell- \sl=\sig'_\ell \Pr_{v^\perp}
\sig_\ell$ . If  $ 1 \le \ell \le \min\{k -m, n-m\}$, then
\be\label{dR-zon} (\R^* \vp)(v)= \Gam_\ell ((n-m)/2)
 (J_+^{\frac{k-m}{2},\, \frac{n-k}{2}}
\vp_0)(r).
 \ee

\end{theorem}

\begin{proof}
By  (\ref{dual-rad}),

\be (\R^* \vp)(v)=\intl_{V_{n-m,n-k}} \vp_0(z_u z'_u) \,
 d_*u, \qquad z_u=\sig'_\ell \gam_v \left[ \begin{array} {c} u \\ 0   \end{array}
\right].\ee
Denote \[ a=\gam^{-1}_v \sig_\ell=\left[ \begin{array} {c} a_1 \\
 a_2
 \end{array} \right], \qquad a_1 \in \frM_{n-m,\ell}, \quad  a_2 \in
 \frM_{m,\ell},
\]
so that  $z_u=a'_1 u$. By  Lemma 2.1, $a_1=wr^{1/2}$, $w\in V_{n-m,\ell}$, and

       \[ r=
      a'_1 a_1=\sig'_\ell \gam_v \left[\begin{array} {cc} I_{n-m} & 0 \\
      0
       & 0 \end{array} \right] \gam_v^{-1} \sig_\ell = \sig'_\ell \Pr_{v^\perp}
       \sig_\ell .
       \]
 Therefore,
\bea    \nonumber
(\R^*
 \vp)(v)&=&\intl_{V_{n-m,n-k}} \vp_0(a'_1uu'a_1) \,
 d_*u\\\nonumber&=&
 \intl_{V_{n-m,n-k}}
 \vp_0(r^{1/2}w'uu'wr^{1/2}) d_*u\\\nonumber&=&\intl_{V_{n-m,n-k}}
                                      \vp_0(r^{1/2}w_0'uu'w_0r^{1/2}) d_*u , \quad  w_0 =
\left[\begin{array} {c} I_\ell\\ 0 \end{array} \right] \in V_{n-m,
      \ell}.  \eea
 Applying  the bi-Stiefel decomposition
  $$
    u= \left[\begin{array} {cc} b \\ u_1(I_{n-k} -b'b)^{1/2}
   \end{array} \right], \qquad b\in \frM_{\ell,n-k}, \quad u_1 \in
 V_{n-m-\ell, n-k},
   $$
   we obtain

  \bea    \nonumber
(\R^* \vp)(v) &=&c \intl_{\{b\in \frM_{\ell,
n-k}:\;0<b'b<I_{n-k}\}}\!\!\!\!\!\!\!\!\!\!\!\!\!\!
\!\!\!\!\!\!\!|I_{n-k} -b'b|^{(k-m-\ell-1)/2} \vp_0 ( r^{1/2}bb'
r^{1/2}) db
\\\nonumber
&=&c\intl_{\{y\in \frM_{ n-k,\ell}:\;0<y'y<I_{\ell}\}}
\!\!\!\!\!\!\!\!\!\!\!\!\!\!|I_\ell -y'y|^{(k-m-\ell-1)/2} \vp_0 (
r^{1/2}y'y r^{1/2}) dy, \eea where $c= \sig_{n-m-\ell,
n-k}/\sig_{n-m,n-k}.$ The change of variables $z=y r^{1/2}$, $dz=|r|^{(n-k)/2} dy$, gives
\bea    \nonumber (\R^* \vp)(v) &=&c |r|^{(k-n)/2}\intl_{\{z\in
\frM_{ n-k,\ell}:\;0<z'z<r\}} \!\!\!\!\!\!\!\!\!\!\!\!\!\!|I_\ell
-r^{-1/2}z'zr^{-1/2}|^{(k-m-\ell-1)/2} \vp_0 ( z'z ) dz
\\\nonumber
&=&c |r|^{(m-n+\ell+1)/2}\intl_{\{z\in \frM_{ n-k,\ell}:\;0<z'z<r\}}\!\!\!\!\!\!\!\!\!\!\!\!\!\! |r
-z'z|^{(k-m-\ell-1)/2} \vp_0 ( z'z ) dz\, .  \eea According to  (\ref{2.22n}), we obtain
$$
      (\R^* \vp)(v)=  c_1
\;   (J_+^{\frac{k-m}{2},\, \frac{n-k}{2}}
\vp_0)(r),
      $$  where
      $$ c_1=\frac{ \pi^{\ell(k-n) /2}\sig_{n-m-\ell,n-k}\Gam_{\ell}(\frac{k-m}{2})}{\sig_{n-m,n-k}}
      =\frac{ \Gam_{\ell}(\frac{k-m}{2})\Gam_{n-k} (\frac{n-m}{2})}{\Gam_{n-k} (\frac{n-m-\ell}{2})}.
      $$
By (\ref{2.5.2}),
$$
 \frac{\Gam_{n-k}(\frac{n-m-\ell}{2})}{\Gam_{n-k} (\frac{n-m}{2})}=  \frac{\Gam_\ell(\frac{k-m}{2})}{\Gam_\ell
 (\frac{n-m}{2})},
$$
and therefore, $c=\Gam_\ell ((n-m)/2))$.

\end{proof}

\bibliographystyle{amsalpha}

\end{document}